\documentclass[12pt,reqno]{amsart}
\usepackage{amsmath,amsfonts,graphicx,amscd,amssymb,amsthm,enumerate,parskip,mathrsfs,color,ltablex,blindtext,cancel,url}  
 \usepackage[a4paper, margin=1in]{geometry}
\usepackage{times}
 
\theoremstyle{definition}
\newtheorem{theorem}{Theorem}[section]
\newtheorem{proposition}[theorem]{Proposition}
\newtheorem{lemma}[theorem]{Lemma}
\newtheorem{corollary}[theorem]{Corollary}
\newtheorem{question}[theorem]{Question}
\newtheorem{definition}[theorem]{Definition}

\newtheorem{remark}[theorem]{Remark}
\newtheorem{introthm}{Theorem}

\numberwithin{equation}{section}
 
\newcommand{\Ric}{\operatorname{Ric}}

\usepackage{cite}
\usepackage[colorlinks,citecolor=blue]{hyperref}

\begin{document}
  
\title{The K\"{a}hler-Ricci Soliton on Bounded Pseudoconvex Domains}
\author{Zehao Sha}

\begin{abstract}
In this paper, we establish rigidity results for K\"ahler--Ricci solitons on bounded pseudoconvex domains in $\mathbb C^n$. We first show that if a complete K\"ahler metric has $C^1$ quasi-bounded geometry and negatively pinched Ricci curvature outside a compact set, then any soliton vector field whose metric dual is closed must vanish. In particular, such a soliton is K\"ahler--Einstein. We then prove that any Bergman K\"ahler--Ricci soliton of bounded domains is K\"ahler--Einstein. Combined with Huang--Xiao's resolution of Cheng's conjecture, this yields a soliton analogue of Cheng's conjecture. We also provide several model domains as examples.
\end{abstract}
\maketitle

\section{Introduction}

In K\"ahler geometry, a central theme is the search for canonical metrics in a given class.  The most classical candidate is the K\"ahler--Einstein metric
\[
        \Ric(\omega)=\lambda \omega,
\]
which has played a fundamental role since the work of Aubin \cite{aubin1978equations}, Yau \cite{yau1978ricci} and many others. From the viewpoint of the Ricci flow, a natural enlargement of this notion is provided by K\"ahler--Ricci solitons.  These are K\"ahler metrics satisfying
\begin{equation}\label{eq:intro-krs}
        \Ric(\omega)+\mathcal L_X\omega=\lambda \omega,
\end{equation}
where \(X\) is a real holomorphic vector field and \(\lambda\in\mathbb R\). They arise as self-similar solutions to the Ricci flow, beginning with the work of Hamilton and Cao \cite{hamilton1988,cao1996existence,cao1997limits}.  A soliton is called trivial if \(X\) is Killing; in this case \eqref{eq:intro-krs} reduces to the K\"ahler--Einstein equation.

On compact K\"ahler manifolds, the existence and rigidity of K\"ahler--Ricci solitons are strongly constrained by the sign of the first Chern class and by Futaki-type obstructions \cite{futaki1983obstruction}.  In contrast, the non-compact case is considerably more flexible: complete non-trivial solitons do exist, and the problem of deciding when a soliton must be trivial becomes much more delicate. 

The purpose of this paper is to study this rigidity problem on bounded pseudoconvex domains. Such domains form a basic class of non-compact complex manifolds with rich boundary geometry. In particular, every bounded pseudoconvex domain carries complete K\"ahler--Einstein metrics, as shown by Cheng--Yau \cite{Cheng1980OnTE} and Mok--Yau \cite{mokyau}.  It is therefore natural to ask which geometric conditions force a K\"ahler--Ricci soliton on such a domain to be trivial.

Our first main result gives a general rigidity result for complete K\"ahler metrics on bounded pseudoconvex domains.  The assumptions are designed to capture the asymptotically negative Ricci curvature behavior that appears naturally near the boundary for many complete metrics.

\begin{introthm}\label{thm:A-general-KRS}
Let \(\Omega\subset\mathbb C^n\) be a bounded pseudoconvex domain with \(C^2\) boundary, and let \(\omega\) be a complete K\"ahler metric on \(\Omega\) of \(C^1\) quasi-bounded geometry.
Assume that there exists a compact subset \(K\Subset\Omega\) such that
\[
-C_2 \omega\le \Ric(\omega)\le -C_1\omega \qquad\text{on } \Omega\setminus K,
\]
for constants \(C_2\ge C_1>0\). If \(\omega\) is a K\"ahler--Ricci soliton with real holomorphic vector field \(X\), and if the dual one-form of \(X\) with respect to \(\omega\) is closed, then \(X=0\). In particular, \(\omega\) is K\"ahler--Einstein.
\end{introthm}

The proof is based on a Bochner formula for real holomorphic vector fields, together with the Omori--Yau maximum principle. The asymptotic negativity of the Ricci tensor first gives a global bound for the soliton vector field. The soliton equation and the Bochner identity then force the maximum of \(|X|^2\) to vanish. The closedness assumption on the dual one-form is automatic for a gradient soliton, but is slightly weaker than requiring a globally defined soliton potential.

We next turn to the Bergman metric. For a bounded domain \(\Omega\subset\mathbb C^n\), the Bergman kernel \(K_\Omega(z,z)\) defines the Bergman metric
\[
        \omega_B=\sqrt{-1}\,\partial\bar\partial\log K_\Omega(z,z),
\]
which is biholomorphically invariant. This metric is one of the most important intrinsic K\"ahler metrics in several complex variables, and its curvature reflects fine boundary geometry of the domain.

The relation between the Bergman metric and the complete K\"ahler--Einstein metric has been a central theme since the work of Cheng--Yau \cite{Cheng1980OnTE} and Mok--Yau \cite{mokyau} on complete K\"ahler--Einstein metrics on bounded pseudoconvex domains. Yau asked, in broad form, whether a bounded pseudoconvex domain with K\"ahler--Einstein Bergman metric should be homogeneous \cite{schoen2010lectures}. On strictly pseudoconvex domains with $C^\infty$-boundary, Cheng formulated the sharper conjecture that such a domain has Bergman K\"ahler--Einstein metric if and only if it is biholomorphic to the unit ball \cite{cheng1979openproblems}. This conjecture was proved in complex dimensions \(n\ge2\) by Huang and Xiao \cite{huang2021bergman}.

The soliton analogue of this problem is the following:
\begin{question}\label{question:intro-Bergman-KRS}
Can the Bergman metric of a bounded pseudoconvex domain carry a non-trivial K\"ahler--Ricci soliton structure?
In particular, on a smoothly bounded strictly pseudoconvex domain, does the Bergman K\"ahler--Ricci soliton condition characterize the unit ball?
\end{question}

Our second main result answers this question on strictly pseudoconvex domains.

\begin{introthm} \label{thm:B-Bergman-KRS}
Let \(\Omega\subset\mathbb C^n\) be a bounded strictly pseudoconvex domain with \(C^\infty\) boundary, and let \(\omega_B\) be its Bergman metric. If \(\omega_B\) is a K\"ahler--Ricci soliton, then \(\omega_B\) is K\"ahler--Einstein. Moreover, \(\Omega\) is biholomorphic to the unit ball.
\end{introthm}

The proof of Theorem~\ref{thm:B-Bergman-KRS} uses an intrinsic mechanism of the Bergman metric on bounded domains. In particular, it is known that the Ricci curvature of a complete Bergman metric on bounded domains is bounded from above \cite{Kobayashi59}. Together with the scalar curvature lower bound for the Ricci soliton \cite{ChenBL09,PRS11,zhang09} (see also \cite{chow2023ricci}), we obtain the uniform bound for the Lie derivative term in the soliton equation. Hence, the soliton vector field generates a one-parameter group of biholomorphisms preserving the Bergman metric (see Theorem~\ref{thm:bergman-krs-ke-general}). Finally, the ball characterization follows from Huang--Xiao's theorem \cite{huang2021bergman} when \(n \ge 2\) and Lu's uniformization theorem \cite{Lu66} when \(n=1\).

In complex dimension one, Topping and Yin constructed complete non-trivial expanding K\"ahler--Ricci solitons on the unit disc \cite[Section~5.3, Example~5.4]{ToppingYin2024}. It is therefore natural to ask which geometric conditions force a K\"ahler--Ricci soliton on a bounded pseudoconvex domain to be trivial. Theorem~\ref{thm:A-general-KRS} gives one such criterion.

\subsection*{Acknowledgements}
The author is deeply grateful to his advisors, G\'erard Besson and Herv\'e Gaussier, for their constant encouragement, guidance, and many valuable discussions. He is also grateful to Alix Deruelle and Vincent Koziarz, who served as referees for his doctoral thesis, for their careful reading and valuable comments on an earlier version of this work. Finally, the author thanks the anonymous referee for a careful reading of the manuscript and for constructive comments and suggestions that led to several improvements.

\section{Preliminaries}
\subsection{Conventions}

Let \((M^n,\omega)\) be a complete K\"ahler manifold.  In local holomorphic coordinates \((z^1,\ldots,z^n)\), we write
\[
\omega=\sqrt{-1}\,g_{i\bar j}\,dz^i\wedge d\bar z^j,\qquad (g^{i\bar j})=(g_{i\bar j})^{-1},
\]
and we use the Einstein summation convention.  All norms, inner products, and contractions are taken with respect to \(\omega\).

We write \(\det(\omega):=\det(g_{i\bar j})\).  The Ricci form is given by
\[
 \Ric(\omega)=\sqrt{-1}\,R_{i\bar j}\,dz^i\wedge d\bar z^j,\qquad R_{i\bar j}=-\partial_i\partial_{\bar j}\log\det(\omega),
\]
and the scalar curvature is
\[
R(\omega)=g^{i\bar j}R_{i\bar j}.
\]
We denote by \(\nabla\) the Levi--Civita connection of \(\omega\).  For a smooth real-valued function \(f\), we use the convention
\[
    \Delta f=g^{i\bar j}\partial_i\partial_{\bar j}f.
\]
Thus \(\Delta\) is the complex Laplacian.  Since \(\omega\) is K\"ahler, the mixed Christoffel symbols vanish in holomorphic coordinates, and hence
\[
    \nabla_i\nabla_{\bar j}f=\partial_i\partial_{\bar j}f.
\]
\subsection{K\"ahler--Ricci solitons}

A K\"ahler metric \(\omega\) is K\"ahler--Einstein if
\begin{equation}\label{eq:KE}
        \Ric(\omega)=\lambda \omega
\end{equation}
for some constant \(\lambda\in\mathbb R\).

Let \(X\) be a real vector field on \(M\).  We say that \(X\) is real holomorphic if its \((1,0)\)-part
\[
        X^{1,0}:=\frac12(X-\sqrt{-1}JX)
\]
is a holomorphic vector field.  A complete K\"ahler metric \(\omega\) is called a K\"ahler--Ricci soliton if there exist a real holomorphic vector field \(X\) and a constant \(\lambda\in\mathbb R\) such that
\begin{equation}\label{eq:KRS}
        \Ric(\omega)+\mathcal L_X\omega=\lambda \omega.
\end{equation}
The soliton is called shrinking, steady, or expanding according as \(\lambda>0\), \(\lambda=0\), or \(\lambda<0\).  It is called trivial if \(X\) is Killing.  In that case \(\mathcal L_X\omega=0\), and \eqref{eq:KRS} reduces to the K\"ahler--Einstein equation \eqref{eq:KE}.

If there exists a smooth real-valued function \(f\) such that \( X=\frac12\nabla f\), then \(\omega\) is called a gradient K\"ahler--Ricci soliton, and \(f\) is called a potential function.  In this case the dual one-form \(X^\flat:=\langle X,\cdot\rangle_\omega\) satisfies \( X^\flat=\frac12\,df\), and hence is closed.
\subsection{Quasi-bounded geometry and the Omori--Yau maximum principle}

To establish our main result, we first introduce the concept of (quasi-)bounded geometry.
\begin{definition}
Let \((M,\omega)\) be a complete K\"ahler manifold and let \(k\ge0\).  We say that \((M,\omega)\) has \(C^k\)-quasi-bounded geometry if, for each \(0\le \ell\le k\), there exists a constant \(C_\ell>0\) such that
\[
        \sup_M |\nabla^\ell \operatorname{Rm}|\le C_\ell .
\]
If, in addition, the injectivity radius of \(\omega\) is bounded from below by a positive constant, then \(\omega\) is said to have \(C^k\)-bounded geometry.
\end{definition}
We shall use the following form of the Omori--Yau maximum principle \cite{omori1967isometric,yau1975harmonic}.
\begin{proposition}[Omori--Yau maximum principle]
\label{prop:OY}
Let \((M,g)\) be a complete Riemannian manifold whose Ricci curvature is bounded from below.  If \(u\in C^2(M)\) satisfies \(\sup_M u<+\infty\), then there exists a sequence \(\{p_j\}\subset M\) such that
\[
u(p_j)\to\sup_M u,\qquad|\nabla u|(p_j)\to0,\qquad\limsup_{j\to\infty}\Delta u(p_j)\le0.
\]
\end{proposition}

\section{Proof of Theorem~A} \label{sec:A}

Let \(\Omega\subset\mathbb C^n\) and \(\omega\) be as in Theorem~\ref{thm:A-general-KRS}.
Assume that $\omega$ is a K\"ahler-Ricci soliton for a real holomorphic vector field \(X\). We first note that the soliton must be expanding.  Indeed, since there is a compact subset \(K\subset \Omega \) such that \(\Ric(\omega)\) is negatively pinched outside \(K\), we have
\[
 R(\omega)\le -nC_1\qquad\text{on }\Omega\setminus K.
\]
Thus the scalar curvature is negative somewhere. On the other hand, by \cite[Theorem~2.14]{chow2023ricci}, every complete Ricci soliton with nonnegative soliton constant has nonnegative scalar curvature. Hence \(\lambda\geq0\) would imply \(R\geq0\) on \(\Omega\), which is a contradiction. Therefore, we have \(\lambda<0\). After replacing \((\omega,X)\) by \(\left(-\lambda\omega,\frac{1}{-\lambda}X\right)\), throughout this section, we may normalize the soliton equation as
\begin{equation}\label{eq:normalized-KRS-A}
     \Ric(\omega)+\mathcal L_X\omega=-\omega.
\end{equation}

We shall use the closedness assumption on the dual one-form \(X^\flat=\langle X,\cdot\rangle_\omega\) in the following case.  Since \(dX^\flat=0\), for any vector fields \(Y,Z\),
\[
        (\nabla_YX^\flat)(Z)=(\nabla_ZX^\flat)(Y).
\]
In local holomorphic coordinates this gives, in particular,
\begin{equation}\label{eq:closed-dual-coordinate}
        \nabla_iX_{\bar j}=\nabla_{\bar j}X_i.
\end{equation}

We also present the following uniqueness observation.

\begin{proposition}\label{lem:soliton-field-unique}
Under the assumptions of Theorem~\ref{thm:A-general-KRS}, the real holomorphic soliton vector field whose dual one-form is closed is unique for a fixed soliton constant.
\end{proposition}

\begin{proof}
Suppose that \(X\) and \(\widetilde X\) are two such vector fields with the same soliton constant, and set \(Y=X-\widetilde X\).  Subtracting the two soliton equations gives
\[
        \mathcal L_Y\omega=0.
\]
Moreover, \(Y\) is real holomorphic and \(Y^\flat\) is closed.  Hence, in local
coordinates,
\[
        \nabla_jY_i=0,\qquad\nabla_{\bar j}Y_i=0.
\]
Applying \(g^{k\bar j}\nabla_k\) to the second identity and commuting covariant derivatives, we obtain
\[
        0=g^{k\bar j}\nabla_k\nabla_{\bar j}Y_i=-R_{i\bar l}Y^{\bar l}.
\]
Thus \(\Ric(Y,JY)=0\).  Since \(\Ric(\omega)\le -C_1\omega\) on \(\Omega\setminus K\), it follows that \(Y=0\) on \(\Omega\setminus K\). By the identity theorem for holomorphic vector fields, \(Y\equiv0\) on \(\Omega\).
\end{proof}

Next we recall the Bochner formula needed below.

\begin{lemma}\label{lem:bochner-real-holomorphic}
Let \(\omega\) be a K\"ahler metric and let \(X\) be a real holomorphic vector field.  Then
\begin{equation}\label{eq:bochner-real-holomorphic}
        \Delta |X|^2=|\nabla X|^2-\Ric(X,JX).
\end{equation}
\end{lemma}

\begin{proof}
We compute this formula in local coordinates. Since \(X\) is real holomorphic, \(\nabla_iX_j=0\). Therefore
\[
\Delta |X|^2=\nabla^i\nabla_i(X^kX_k)=\nabla^iX_k\nabla_iX^k+X_k\nabla^i\nabla_iX^k.
\]
Commuting derivatives gives
\[
\nabla^i\nabla_iX^k=-R^{\ k}_{\ell}X^\ell,
\]
and hence
\[
\Delta |X|^2=|\nabla X|^2-\Ric(X,JX),
\]
which completes the proof.
\end{proof}

We now prove Theorem~\ref{thm:A-general-KRS}.

\begin{proof}[Proof of Theorem~\ref{thm:A-general-KRS}]
By the normalized soliton equation \eqref{eq:normalized-KRS-A}, we have, in local coordinates
\begin{equation}\label{eq:soliton-local-A}
        R_{i\bar j}+\nabla_iX_{\bar j}+\nabla_{\bar j}X_i=-g_{i\bar j}.
\end{equation}
Using \eqref{eq:closed-dual-coordinate}, this becomes
\[
R_{i\bar j}+2\nabla_{\bar j}X_i=-g_{i\bar j}.
\]
Applying \(g^{k\bar j}\nabla_k\) and using the contracted Bianchi identity, we obtain
\[
\begin{aligned}
        0&=g^{k\bar j}\nabla_kR_{i\bar j}+2g^{k\bar j}\nabla_k\nabla_{\bar j}X_i  \\
        &=\nabla_iR+2g^{k\bar j}\left( \nabla_{\bar j}\nabla_kX_i-R_{k\bar j i\bar l}X^{\bar l}\right).
\end{aligned}
\]
Since \(X\) is real holomorphic, \(\nabla_kX_i=0\).  Hence
\begin{equation}\label{eq:gradR-RicX}
        \nabla_iR=2R_{i\bar l}X^{\bar l}.
\end{equation}
Equivalently,
\[
\langle\nabla R,\cdot\rangle_\omega=2\Ric(X,J\cdot).
\]

We first prove that \(X\) is bounded.  On \(\Omega\setminus K\), we have
\[
        \Ric(X,JX)\le -C_1|X|^2.
\]
Taking the pairing of \eqref{eq:gradR-RicX} with \(X\), we obtain
\[
\langle\nabla R,X\rangle_\omega=2\Ric(X,JX)\le -2C_1|X|^2.
\]
Therefore, on \(\Omega\setminus K\),
\[
2C_1|X|^2\le-\langle\nabla R,X\rangle_\omega\le|\nabla R|\,|X|.
\]
It follows that
\[
|X|\le\frac{|\nabla R|}{2C_1}\qquad\text{on }\Omega\setminus K.
\]
Since \(\omega\) has \(C^1\) quasi-bounded geometry and \(X\) is bounded on the compact set \(K\), we conclude that
\begin{equation}\label{eq:X-bounded-A}
        \sup_\Omega |X|<+\infty.
\end{equation}

Next, evaluating the normalized soliton equation \eqref{eq:normalized-KRS-A} on \((X,JX)\), we obtain
\[
    \Ric(X,JX)+(\mathcal L_X\omega)(X,JX)=-|X|^2.
\]
Since
\[
(\mathcal L_X\omega)(X,JX)=\nabla_X|X|^2,
\]
we get
\begin{equation}\label{eq:RicXX-Xu}
 \Ric(X,JX) = -\nabla_X|X|^2-|X|^2.
\end{equation}
Combining \eqref{eq:bochner-real-holomorphic} and \eqref{eq:RicXX-Xu}, we find
\[ 
\Delta |X|^2=|\nabla X|^2+\nabla_X|X|^2+|X|^2\ge\nabla_X|X|^2+|X|^2.
\]

Now apply the Omori--Yau maximum principle to the bounded function \(u:=|X|^2\). There exists a sequence \(\{p_j\}\subset\Omega\) such that
\[
u(p_j)\to\sup_\Omega u,\qquad|\nabla u|(p_j)\to0,\qquad\limsup_{j\to\infty}\Delta u(p_j)\le0.
\]
Using \eqref{eq:X-bounded-A}, we have \(\nabla_Xu(p_j)\to0\).
Therefore,
\[
0\ge\limsup_{j\to\infty}\Delta u(p_j)\ge\lim_{j\to\infty}\nabla_Xu(p_j)+\lim_{j\to\infty}u(p_j)=\sup_\Omega u.
\]
Hence \(\sup_\Omega |X|^2=0\), which implies \(X\equiv0\). Consequently, \(\omega\) is K\"ahler--Einstein.
\end{proof}
\section{Bergman metrics and metrics induced by defining functions} \label{sec:Bergman-and-defining}

In this section we discuss two natural types of complete K\"ahler metrics on bounded domains.

\subsection{The Bergman metric} \label{subsec:Bergman-metric}

Let \(\Omega\subset\mathbb C^n\) be a bounded domain and let
\[
A^2(\Omega)=\{f\in\mathcal O(\Omega): \int_\Omega |f|^2\,dV<+\infty\}
\]
be its Bergman space.  Denote by \(K_\Omega(z,w)\) the Bergman kernel and write
\[
K_\Omega(z):=K_\Omega(z,z).
\]
The Bergman metric is the K\"ahler metric
\[
\omega_B=\sqrt{-1}\,\partial\bar\partial\log K_\Omega(z).
\]
Note that the Bergman metric is biholomorphically invariant. In particular, every automorphism of \(\Omega\) acts by isometries of \(\omega_B\).

We first recall a simple completeness criterion for vector fields.

\begin{lemma}\label{lem:bounded-liederivative-complete}
Let \((M,g)\) be a complete Riemannian manifold and let \(X\) be a smooth vector field on \(M\). If \(\|\mathcal L_Xg\|_g\le C\) for some constant \(C>0\), then \(X\) is complete.
\end{lemma}

\begin{proof} 
Let \( \gamma:(a,b)\to M \) be a maximal integral curve of \(X\), so that \(\dot\gamma(t)=X(\gamma(t))\). Along \(\gamma\), we have 
\[ 
\frac{d}{dt}|X|_g^2(\gamma(t)) = X\bigl(g(X,X)\bigr)(\gamma(t))=(\mathcal L_Xg)(X,X)(\gamma(t)). 
\] 
Therefore, 
\[ 
\left| \frac{d}{dt}|X|_g^2(\gamma(t)) \right| \le C|X|_g^2(\gamma(t)). 
\] 
By Gronwall's inequality, \(|X|_g(\gamma(t))\) remains bounded up to any finite endpoint of the maximal interval of existence. Suppose, for contradiction, that \(b<+\infty\). Fix \(t_0\in(a,b)\). Since \(|X|_g(\gamma(t))\) is bounded on \([t_0,b)\), the curve \(\gamma\) has finite length on \([t_0,b)\). In particular, \(\gamma(t)\) is a Cauchy curve as \(t\to b^{-}\). Since \((M,g)\) is complete, there exists a point \(p\in M\) such that \( \gamma(t)\to p \) as \(t\to b^{-}\). Since \(X\) is smooth, the ordinary differential equation \(\dot\eta=X(\eta)\) has a local solution starting from \(p\). Thus \(\gamma\) can be extended beyond \(b\), contradicting the maximality of \((a,b)\). Therefore \(b=+\infty\). Applying the same argument backward in time gives \(a=-\infty\). Hence \(X\) is complete. 
\end{proof}

We now prove the rigidity argument for the Bergman K\"ahler-Ricci soliton.

\begin{theorem}\label{thm:bergman-krs-ke-general}
Let \(\Omega\subset\mathbb C^n\) be a bounded domain. If the Bergman metric \(\omega_B\) is a K\"ahler-Ricci soliton associated with a real holomorphic vector field \(X\) and a constant \(\lambda \in \mathbb R\), then \(X\) is Killing.  Consequently, \(\omega_B\) is K\"ahler--Einstein.
\end{theorem}

\begin{proof}
It follows from \cite[Theorem 4.4]{Kobayashi59} that \(\Ric(\omega_B) \le (n+1)\omega_B\). On the other hand, by \cite[Theorem~2.14]{chow2023ricci}, \(R(\omega_B) \ge n \min\{\lambda,0\}\). At a point \(p \in \Omega\), choose a local holomorphic coordinate so that \(\Ric(\omega_B)\) is diagonalized with eigenvalue \(\mu_i\) for each \(1 \le i \le n\). Then
\[ \mu_i = R(\omega_B) - \sum_{k\neq i} \mu_k \ge n \min\{\lambda,0\}-(n-1)(n+1).\]
Hence, \(\Ric(\omega_B)\) is bounded from below. Therefore, we have
\begin{equation}\label{eq:lie-Bergman-bound}
    \|\mathcal L_X\omega_B\|_{\omega_B}\le C\qquad \text{on }\Omega.
\end{equation}
Since \(\omega_B\) is complete, Lemma~\ref{lem:bounded-liederivative-complete} implies that \(X\) is complete.

Let \(\Phi_t\) be the complete flow generated by \(X\).  Since \(X\) is real holomorphic, \(\Phi_t\) is a one-parameter group of biholomorphisms of \(\Omega\).  By the biholomorphic invariance of the Bergman metric,
\begin{equation} \label{eq:Bergman-flow}
\Phi_t^*\omega_B=\omega_B\qquad\text{for all }t\in\mathbb R.    
\end{equation}
Differentiating \eqref{eq:Bergman-flow} at \(t=0\), we obtain \(\mathcal L_X\omega_B=0\), which implies \(X\) is Killing. Consequently, \(\omega_B\) is K\"ahler--Einstein.
\end{proof}

Combining Theorem~\ref{thm:bergman-krs-ke-general} with \cite[Theorem 1.1]{huang2021bergman}, we complete the proof of Theorem~\ref{thm:B-Bergman-KRS} for \(n \ge2\). For \(n=1\), by Theorem~\ref{thm:bergman-krs-ke-general}, \(\omega_B\) is a hyperbolic metric. Hence Theorem~\ref{thm:B-Bergman-KRS} follows from Lu's uniformization theorem \cite{Lu66} directly.

\subsection{Metrics induced by defining functions}
\label{subsec:defining-function-metrics}

We now turn to a different class of complete K\"ahler metrics. Let \(\Omega\subset\mathbb C^n\) be a bounded strictly pseudoconvex domain with smooth boundary, and let \(\rho\in C^\infty(\overline\Omega)\) be a strictly plurisubharmonic defining function so that
\[
\Omega=\{\rho<0\},\qquad d\rho\neq0\quad\text{on }\partial\Omega,\qquad\sqrt{-1}\,\partial\bar\partial\rho>0\quad\text{near }\overline\Omega.
\]
Define \(\omega_\rho:=-\sqrt{-1}\,\partial\bar\partial\log(-\rho)\). Since
\[
\omega_\rho=\frac{\sqrt{-1}\,\partial\bar\partial\rho}{-\rho}+\frac{\sqrt{-1}\,\partial\rho\wedge\bar\partial\rho}{\rho^2},
\]
the metric \(\omega_\rho\) is complete. For a smooth real-valued function \(u\), the Fefferman operator \cite{fefferman1976monge} is defined by
\[
\mathcal J(u):=-\det
\begin{pmatrix}
    u & u_{\bar j}\\
    u_i & u_{i\bar j}
\end{pmatrix}.
\]
A block determinant computation gives, wherever \(u<0\),
\begin{equation}\label{eq:Fefferman-determinant}
        \det\left(\partial_i\partial_{\bar j}[-\log(-u)]\right)=(-u)^{-(n+1)}\mathcal J(u).
\end{equation}
Since \(\omega_\rho\) is K\"ahler, \eqref{eq:Fefferman-determinant} implies that \(\mathcal J(\rho)>0\) in \(\Omega\). On \(\partial\Omega\), the strict pseudoconvexity of the boundary and \(d\rho\neq0\) imply that \(\mathcal J(\rho)>0\). Consequently, \(\mathcal J(\rho)\) is positive on \(\overline\Omega\). Applying \eqref{eq:Fefferman-determinant} to \(\rho\), we obtain
\begin{equation}\label{eq:defining-log-det-preliminary}
\det(\omega_\rho)=(-\rho)^{-(n+1)}\mathcal J(\rho).
\end{equation}
Consequently,
\begin{equation}\label{eq:defining-Ricci-preliminary}
\Ric(\omega_\rho)=-(n+1)\omega_\rho-\sqrt{-1}\,\partial\bar\partial\log\mathcal J(\rho).
\end{equation}
Since \(\mathcal J(\rho)\) is positive and smooth on \(\overline\Omega\), the form \(\sqrt{-1}\,\partial\bar\partial\log\mathcal J(\rho)\) has bounded coefficients. On the other hand, the inverse coefficients of \(\omega_\rho\) are \(O(-\rho)\) near the boundary. It follows that
\begin{equation}\label{eq:defect-decay}
\left\|\Ric(\omega_\rho)+(n+1)\omega_\rho\right\|_{\omega_\rho}\longrightarrow0\qquad\text{as }z\longrightarrow\partial\Omega.
\end{equation}

Let \(r\in C^\infty(\Omega)\cap C^2(\overline\Omega)\) be the Cheng--Yau defining function of the smoothly bounded strictly pseudoconvex domain \(\Omega\), normalized by \(\mathcal J(r)=1\) \cite{Cheng1980OnTE}. The corresponding metric \(\omega_{\mathrm{KE}}=-\sqrt{-1}\,\partial\bar\partial\log(-r)\) gives the unique complete K\"ahler--Einstein metric satisfying \(\Ric(\omega_{\mathrm{KE}})=-(n+1)\omega_{\mathrm{KE}}\). Let \(\rho\) be any other smooth defining function of \(\Omega\) with \(\omega_\rho=-\sqrt{-1}\,\partial\bar\partial\log(-\rho)\). If \(\omega_\rho\) is K\"ahler, by \eqref{eq:Fefferman-determinant}, we have
\begin{equation}\label{eq:def-vol-ratio}
    \frac{\omega_\rho^n}{\omega_{\mathrm{KE}}^n}=\left(\frac{r}{\rho}\right)^{n+1}\mathcal J(\rho).
\end{equation}
Since both \(r\) and \(\rho\) are defining functions of \(\Omega\), there is a positive \(C^1\) function \(a\) on \(\overline\Omega\) such that \(r=a\rho\). By \cite[Proposition~1.6]{robin1987higher} (or a direct computation), we have
\[
    a^{n+1}\mathcal J(\rho)=\mathcal J(a\rho)=\mathcal J(r)=1\qquad\text{on }\partial\Omega.
\]
Thus \(a^{n+1}\mathcal J(\rho)\) extends continuously to \(\overline\Omega\) with boundary value one. In particular, we will see that if \(\omega_\rho\) is a K\"ahler--Ricci soliton, then \(\omega_\rho^n=\omega_{\mathrm{KE}}^n\) on \(\Omega\).

\begin{proposition} \label{prop:defining-function-krs-rigidity}
Let \(\Omega\subset\mathbb C^n\) be a bounded strictly pseudoconvex domain with \(C^\infty\) boundary, and let \(\rho\in C^\infty(\overline\Omega)\) be a defining function such that \(\omega_\rho=-\sqrt{-1}\,\partial\bar\partial\log(-\rho)\) is a complete K\"ahler metric. Suppose that
\begin{equation}\label{eq:soliton_rho}
        \Ric(\omega_\rho)+\mathcal L_X\omega_\rho=\lambda \omega_\rho
\end{equation}
for a real holomorphic vector field \(X\) and a constant \(\lambda\in\mathbb R\). Then \(\omega_\rho\) is the unique complete K\"ahler--Einstein metric on \(\Omega\) normalized by \(\Ric(\omega_{\mathrm{KE}})= -(n+1)\omega_{\mathrm{KE}}\).
\end{proposition}

\begin{proof}
Write \(F=\log\mathcal J(\rho)\) and \(u=\log(\rho/r)\). By \eqref{eq:defect-decay}, \(R(\omega_\rho)\to-n(n+1)\) as \(z\to\partial\Omega\). Hence \(\lambda<0\) by \cite[Theorem~2.14(a)]{chow2023ricci}. It follows from \eqref{eq:def-vol-ratio} that
\[
v:=\log\frac{\omega_\rho^n}{\omega_{\mathrm{KE}}^n}=F-(n+1)u.
\]
Thus \(v\) extends continuously to \(\overline\Omega\) with zero boundary value. The uniqueness of \(\omega_{\mathrm{KE}}\) implies that for every \(\Phi \in \operatorname{Aut}(\Omega)\), we have \(\Phi^*\omega_{\mathrm{KE}}=\omega_{\mathrm{KE}}\). By \eqref{eq:soliton_rho} and \eqref{eq:defect-decay}, \(\|\mathcal L_X\omega_\rho\|_{\omega_\rho}\) is bounded, so \(X\) is complete by Lemma~\ref{lem:bounded-liederivative-complete}. Since \(X\) is real holomorphic, its flow \(\Phi_t \in \operatorname{Aut}(\Omega)\). Hence we have \(\mathcal L_X\omega_{\mathrm{KE}}=0\). Therefore,
\begin{equation} \label{eq:defining1}
    \mathcal L_X (\omega_\rho^n)= \mathcal L_X \left(e^{v} \omega_{\mathrm{KE}}^n\right)= Xv\cdot e^v\omega_{\mathrm{KE}}^n.
\end{equation}
We also note that, by \eqref{eq:soliton_rho}, we have
\begin{equation} \label{eq:defining2}
     \mathcal L_X (\omega_\rho^n) = n\mathcal L_X \omega_\rho \wedge \omega_\rho^{n-1} = (n\lambda - R(\omega_\rho))\omega_\rho^n.
\end{equation}
Comparing \eqref{eq:defining1} and \eqref{eq:defining2}, by \eqref{eq:def-vol-ratio}, we obtain
\begin{equation}
    Xv = n\lambda - R(\omega_\rho).
\end{equation}
If \(v\equiv0\), we have \(R(\omega_\rho)\equiv n\lambda \) in \(\Omega\). If \(v\not\equiv0\), since \(v=0\) on \(\partial \Omega\), it attains a positive maximum or a negative minimum in \(\Omega\) where \(R(\omega_\rho)=n\lambda\). In either case there is \(p\in\Omega\) such that \(R(\omega_\rho)(p)=n\lambda\). By \cite[Theorem~2.14]{chow2023ricci}, \(\omega_\rho\) is K\"ahler--Einstein with Einstein constant \(\lambda=-(n+1)\). This completes the proof thanks to the uniqueness of normalized K\"ahler--Einstein on \(\Omega\).
\end{proof}

\section{Examples}
\label{sec:examples}
\subsection{Thullen domains in \texorpdfstring{$\mathbb C^2$}{C\string^2}}
\label{subsec:Thullen}

Let \(m\in\mathbb N\), \(m\ge1\), and consider the Thullen domain
\[
\Omega_m:=\{(z,w)\in\mathbb{C}^2: |z|^2+|w|^{2m}<1\},
\]
which is a bounded pseudoconvex complete Reinhardt domain.  When \(m=1\), \(\Omega_m\) is the unit ball.  When \(m>1\), the boundary is weakly pseudoconvex precisely along
\[
 \{(z,0): |z|=1\}\subset\partial\Omega_m .
\]
For \(m>1\), the function
\[
-\log(1-|z|^2-|w|^{2m})
\]
does not define a K\"ahler metric on all of \(\Omega_m\).  We instead use the positive function
\[
\rho(z,w):=(1-|z|^2)^{1/m}-|w|^2
\]
and the complete K\"ahler metric \(\omega_m=-\sqrt{-1}\,\partial\bar\partial\log\rho\) introduced in this context in \cite{seo2012theorem}.  Explicitly, its coefficient matrix is
\[
(g_{m,i\bar j})=\frac{(1-|z|^2)^{\frac1m-2}}{m\rho^2}
\begin{pmatrix}
\rho+\frac1m |z|^2|w|^2&w\bar z(1-|z|^2)\\[1mm]
\bar w z(1-|z|^2)&m(1-|z|^2)^2
\end{pmatrix}.
\]
Its determinant is
\[
\det(\omega_m)=\frac1m\,\frac{(1-|z|^2)^{\frac2m-2}}{\rho^3}.
\]
Consequently,
\[
\Ric(\omega_m)=-\sqrt{-1}\,\partial\bar\partial\log\det(\omega_m),
\]
and a direct computation gives its coefficient matrix
\[
(R_{m,i\bar j})=-\frac{3(1-|z|^2)^{\frac1m-2}}{m\rho^2}
\begin{pmatrix}
\frac23(m-1)\rho^2(1-|z|^2)^{-\frac1m}+\rho+\frac1m |z|^2|w|^2&w\bar z(1-|z|^2)\\[1mm]
\bar w z(1-|z|^2)&m(1-|z|^2)^2
\end{pmatrix}.
\]
In particular, \(\omega_m\) is K\"ahler--Einstein if and only if \(m=1\).

We first show that \(\omega_m\) admits no non-trivial K\"ahler--Ricci soliton structure.

\begin{proposition}\label{prop:Thullen-gm-KRS}
If \(\omega_m\) is a K\"ahler--Ricci soliton on \(\Omega_m\), then \(m=1\).
\end{proposition}

\begin{proof}
Assume that
\[
\Ric(\omega_m)+\mathcal L_X\omega_m=\lambda \omega_m
\]
for a real holomorphic vector field
\[
X=X^1\partial_z+X^2\partial_w+\overline{X^1}\partial_{\bar z}+\overline{X^2}\partial_{\bar w},
\]
where \(X^1\) and \(X^2\) are holomorphic functions on \(\Omega_m\). We use the local expression
\[
(\mathcal L_X\omega_m)_{i\bar j}=X^k\partial_k g_{m,i\bar j}+\overline{X^l}\partial_{\bar l}g_{m,i\bar j}+(\partial_iX^k)g_{m,k\bar j}+(\partial_{\bar j}\overline{X^l})g_{m,i\bar l}.
\]
Applying this to the \((1,\bar1)\)-component and evaluating at \((0,w)\), we obtain
\[
(\mathcal L_X\omega_m)_{1\bar1}(0,w)=\frac{\bar w X^2+w\overline{X^2}}{m(1-|w|^2)^2}+\frac{\partial_z X^1+\partial_{\bar z}\overline{X^1}}{m(1-|w|^2)},
\]
where the coefficients of \(X\) and their derivatives are evaluated at \((0,w)\). The preceding formulas for the metric and its Ricci tensor give
\[
g_{m,1\bar1}(0,w)=\frac{1}{m(1-|w|^2)},\qquad R_{m,1\bar1}(0,w)=-\frac{3}{m(1-|w|^2)}-\frac{2(m-1)}{m}.
\]
Multiplying the \((1,\bar1)\)-component of the soliton equation by \(m(1-|w|^2)^2\), we obtain
\begin{equation}\label{eq:Thullen-11-reduced}
\begin{aligned}
2(m-1)(1-|w|^2)^2+(\lambda+3)(1-|w|^2)&= \bar w X^2+w\overline{X^2} \\
&\quad +(1-|w|^2)\bigl(\partial_z X^1+\partial_{\bar z}\overline{X^1}\bigr).
\end{aligned}
\end{equation}
Since \(X^2(0,w)\) and \((\partial_z X^1)(0,w)\) are holomorphic in \(w\), the coefficient of \(w^2\bar w^2\) on the right-hand side is zero. Applying \(\partial_w^2\partial_{\bar w}^2\) to \eqref{eq:Thullen-11-reduced} and evaluating at \(w=0\) therefore gives \(8(m-1)=0\). Hence \(m=1\).
\end{proof}

We next compare \(\omega_m\) with the Bergman metric.  The Bergman kernel of \(\Omega_m\) is explicitly given by
\begin{equation}\label{eq:Thullen-Bergman-kernel}
K_{\Omega_m}(z,w)=\frac{1}{m\pi^2}\,\frac{(m+1)(1-|z|^2)^{1/m}-(m-1)|w|^2}{\rho^3(1-|z|^2)^{2-\frac1m}}.
\end{equation}
This normalization agrees with the standard Bergman kernel of the unit ball when \(m=1\).

\begin{proposition}\label{prop:Thullen-proportional}
If there exists a constant \(c>0\) such that \(\omega_B=c\,\omega_m\), then \(m=1\).
\end{proposition}

\begin{proof}
Assume \(\omega_B=c\omega_m\).  Since
\[
\omega_B=\sqrt{-1}\,\partial\bar\partial\log K_{\Omega_m},\qquad\omega_m=-\sqrt{-1}\,\partial\bar\partial\log\rho,
\]
there exists a pluriharmonic function \(h\) such that
\[
\log K_{\Omega_m}+c\log\rho=h.
\]
The left-hand side depends only on \(|z|\) and \(|w|\).  Hence \(h\) is a smooth rotation-invariant pluriharmonic function on the complete Reinhardt domain \(\Omega_m\).  Therefore \(h\) is constant.

Using \eqref{eq:Thullen-Bergman-kernel}, we get
\[
\frac{(m+1)(1-|z|^2)^{1/m}-(m-1)|w|^2}{\rho^{3-c}(1-|z|^2)^{2-\frac1m}}=C
\]
for some constant \(C>0\).  Approaching a boundary point at which the numerator does not vanish forces \(c=3\).  Hence
\[
 \frac{(m+1)(1-|z|^2)^{1/m}-(m-1)|w|^2}{(1-|z|^2)^{2-\frac1m}}
\]
is constant.  Setting \(w=0\), this becomes
\[
(m+1)(1-|z|^2)^{\frac2m-2}=\text{constant}.
\]
Therefore \(\frac2m-2=0\), and hence \(m=1\).
\end{proof}

We also have the following consequence of the explicit Bergman kernel computation.

\begin{proposition}\label{prop:Thullen-Bergman-KE}
The Bergman metric \(\omega_B\) of \(\Omega_m\) is K\"ahler--Einstein if and only if \(m=1\).
\end{proposition}

\begin{proof}
The ``if'' part follows from the case of the ball. For the ``only if'' part, we use the same notation as in \cite{azukawa1983bergman}. It is enough to show this at $(0,w)$ with $|w|<1$. Let
    $$
t=\frac{1-|w|^2}{1-r|w|^2}, \qquad |w|<1,
    $$
where $r=(m-1)/(m+1)$. Note that $0\leq r <1$ and $0<t\leq 1$. Then the coefficient matrix of the Bergman metric is given by
$$
(g_{B,i\bar j})(0,w)=\begin{pmatrix}
    \frac{\alpha}{(1+r)t} & 0 \\
    0 & \frac{\beta(1-rt)^2}{(1-r)^2t^2}
\end{pmatrix}
$$
and the coefficient matrix of the Ricci curvature of the Bergman metric is given by
$$
(R_{i\bar j})(0,w)=\begin{pmatrix}
    \frac{3\alpha^2\beta-4A\beta-2B\alpha}{\alpha \beta (1+r)t} & 0 \\
    0 & \frac{(1-rt)^2\left(3\alpha \beta^2 -4C \alpha -2B\beta\right)}{\alpha \beta (1-r)^2 t^2}
\end{pmatrix},
$$
where
\begin{equation}
\begin{cases}\label{th}
\alpha=3+r t^2, \quad \beta=3-r t^2 , \\
A=6+4 r t^2+(1+r) r t^3, \\
B=2\left(9+3 r t^2-3(1+r) r t^3+2 r^2 t^4\right) /\left(3+r t^2\right), \\
C=3\left(6-6 r t^2+(1+r) r t^3\right) /\left(3-r t^2\right) .
\end{cases} 
\end{equation}

Assume that the Bergman metric is K\"ahler--Einstein. At \(w=0\) we have \(t=1\), \(A=(r+2)\alpha\), and \(B=2\beta\). Substitution into the \((1,\bar1)\)-component gives \(R_{1\bar1}(0,0)=-g_{B,1\bar1}(0,0)\). Thus the Einstein constant is \(-1\), and hence
$$
(R_{i\bar j})(0,w) =-(g_{B,i\bar j})(0,w).
$$
Rearranging the \((1,\bar 1)\)-component of the Ricci curvature yields
\begin{align}\label{thbe}
    4\alpha^2\beta -4A\beta -2B\alpha =0.
\end{align}
Plugging \eqref{th} into \eqref{thbe}, we get
$$
r^2 [rt^2-(1+r)t+1]=0.
$$
Since $0\leq r <1$, for \(0<t<1\), we have
\[
rt^2-(1+r)t+1=(1-t)(1-rt)>0.
\]
Since the equation holds for all \(|w|<1\), we may choose \(w\ne0\), hence \(0<t<1\). Therefore \(r=0\).
\end{proof}

Combining the preceding results, we obtain the following characterization.

\begin{theorem}\label{thm:Thullen-characterization}
Let
\[
\Omega_m=\{(z,w)\in\mathbb C^2:\ |z|^2+|w|^{2m}<1\},\qquad m\in\mathbb N_{\ge1},
\]
be the Thullen domain. Then the following are equivalent:
\begin{enumerate}
        \item \(m=1\), equivalently \(\Omega_m\) is the unit ball;
        \item \(\omega_m\) is K\"ahler--Einstein, equivalently \(\omega_m\) is a K\"ahler--Ricci soliton;
        \item \(\omega_B\) is K\"ahler--Einstein, equivalently \(\omega_B\) is a K\"ahler--Ricci soliton;
        \item there exists a constant \(c>0\) such that \(\omega_B=c\omega_m\).
\end{enumerate}
\end{theorem}

\begin{proof}
The equivalence between (1) and the K\"ahler--Einstein condition for \(\omega_m\) follows from the explicit formula for \(\Ric(\omega_m)\), while Proposition~\ref{prop:Thullen-gm-KRS} shows that a K\"ahler--Ricci soliton structure on \(\omega_m\) forces \(m=1\).  Thus (1) and (2) are equivalent.

Similarly, Proposition~\ref{prop:Thullen-Bergman-KE} gives the equivalence between (1) and the K\"ahler--Einstein condition for \(\omega_B\).  If \(\omega_B\) is a K\"ahler--Ricci soliton, then Theorem~\ref{thm:bergman-krs-ke-general} shows that its soliton vector field is Killing. Consequently, \(\omega_B\) is K\"ahler--Einstein.  Hence (1) and (3) are equivalent.

Finally, Proposition~\ref{prop:Thullen-proportional} shows that \(\omega_B=c\omega_m\) can occur only when \(m=1\).  Conversely, when \(m=1\), \(\Omega_m\) is the unit ball, and both \(\omega_m\) and \(\omega_B\) are constant multiples of the complex hyperbolic metric.  Hence (1) and (4) are equivalent.
\end{proof}

\subsection{Strictly pseudoconvex Hartogs domains} \label{subsec:Hartogs}

Let \(n\ge2\) and \(x_0>0\). Consider \(F\in C^\infty([0,x_0])\) satisfy
\[
F(x)>0\quad(0\le x<x_0),\qquad F(x_0)=0,\qquad F'(x)<0\quad(0\le x\le x_0),
\]
and
\begin{equation}\label{eq:Hartogs-strict-pseudoconvexity}
        \frac{d}{dx}\left(\frac{xF'(x)}{F(x)}\right)<0
        \qquad(0\le x<x_0).
\end{equation}
The associated Hartogs domain is
\[
D_F:=\left\{(z_0,z_1,\ldots,z_{n-1})\in\mathbb C^n:|z_0|^2<x_0,\ \sum_{j=1}^{n-1}|z_j|^2<F(|z_0|^2)\right\}.
\]
Set
\[
A_F(z):=F(|z_0|^2)-\sum_{j=1}^{n-1}|z_j|^2,\qquad\rho_F:=-A_F,
\]
and define
\[
\omega_F:=-\sqrt{-1}\,\partial\bar\partial\log A_F.
\]
Condition \eqref{eq:Hartogs-strict-pseudoconvexity} is precisely the condition that \(\omega_F\) be positive \cite[Proposition~2.1]{loizuddas1}.  Moreover, \(D_F\) is bounded with a smooth strictly pseudoconvex boundary and \(\rho_F\) is a global smooth defining function. The following characterization is a consequence of Proposition \ref{prop:defining-function-krs-rigidity}.

\begin{corollary}\label{cor:Hartogs-KRS}
Let \(D_F\) and \(\omega_F\) be as above.  Then the following are equivalent:
\begin{enumerate}
        \item \(\omega_F\) is a K\"ahler--Ricci soliton;
        \item \(\omega_F\) is K\"ahler--Einstein;
        \item there exist constants \(C_1,C_2>0\) such that \(F(x)=C_1-C_2x\);
        \item \((D_F,\omega_F)\) is holomorphically isometric to complex hyperbolic space.
\end{enumerate}
\end{corollary}

\begin{proof}
The implication \((1)\Rightarrow(2)\) follows from Proposition~\ref{prop:defining-function-krs-rigidity}, whose hypotheses hold for the global defining function \(\rho_F\). If \(\omega_F\) is K\"ahler--Einstein, then its scalar curvature is constant.  By \cite[Lemma~2.3]{loizuddas1}, this forces \(F(x)=C_1-C_2x\) for constants \(C_1,C_2>0\).  The linear map
\[
(z_0,z_1,\ldots,z_{n-1})\longmapsto\left(\sqrt{\frac{C_2}{C_1}}\,z_0,\frac{z_1}{\sqrt{C_1}},\ldots,\frac{z_{n-1}}{\sqrt{C_1}}\right)
\]
maps \(D_F\) biholomorphically onto the unit ball and identifies \(\omega_F\) with its complex hyperbolic metric. This proves \((2)\Rightarrow(3)\Rightarrow(4)\).  Finally, the complex hyperbolic metric is K\"ahler--Einstein and hence is a trivial K\"ahler--Ricci soliton, proving \((4)\Rightarrow(1)\).
\end{proof}

\begin{remark}
For \(n\ge2\), Loi--Zuddas proved by an explicit coordinate computation that every K\"ahler--Ricci soliton of this form is K\"ahler--Einstein \cite[Theorem~1.2]{loizuddas1}. Their more general statement concludes that \((D_F,\omega_F)\) is an open subset of complex hyperbolic space. Under the present endpoint condition \(F(x_0)=0\), that open subset is the whole ball.
\end{remark}

\bibliographystyle{amsalpha}
\bibliography{wpref}

\bigskip
  \footnotesize

  Zehao Sha, \textsc{C. N. Yang Institute for Mathematics and Fundamental Physics, Hefei, China}\par\nopagebreak
  Email address: \texttt{zhsha@imfp.org.cn}\par\nopagebreak
  Homepage: \url{https://ricciflow19.github.io/}
\end{document}